\documentclass[a4paper, 11pt]{article}

\usepackage{amsmath}
\usepackage{amssymb}
\usepackage{amsthm}
\usepackage{mathrsfs}
\usepackage{array}
\usepackage{enumerate}
\usepackage{enumitem}
\setlist{noitemsep}
\usepackage{scrextend}
\usepackage{graphicx}
\usepackage{hyperref}
\def\equationautorefname~#1\null{(#1)\null}
\usepackage[margin=1in]{geometry}
\usepackage{tikz}
\usepackage{tkz-euclide}
\usetkzobj{all} 
\usepackage{dsfont}
\usepackage{cite}
\usepackage[font={small,it}]{caption}
\usepackage{standalone}
\usepackage{algorithm}
\usepackage[noend]{algpseudocode}
\usepackage{color}

\makeatletter
\g@addto@macro\@floatboxreset\centering
\makeatother

\newtheorem{prop}{Proposition}[section]
\newtheorem{thm}[prop]{Theorem}

\newtheorem{lemma}[prop]{Lemma}
\newtheorem{conj}[prop]{Conjecture}
\theoremstyle{definition}

\newcommand{\D}{\mathcal{D}_R}
\newcommand{\Dis}[1]{\mathcal{D}_{#1}}

\newcommand{\N}{\mathbb{N}}

\newcommand{\R}{\mathbb{R}}

\newcommand{\An}[1]{\mathcal{A}_{#1}}

\newcommand{\E}{\mathbb{E}}
\newcommand{\pr}{\mathbb{P}}

\newcommand{\PP}{\mathcal{P}_{n,\alpha, \nu}}
\newcommand{\im}{\mu_{n,\alpha, \nu}}
\newcommand{\nn}{\nonumber}
\newcommand{\eqd}{\stackrel{\Delta}{=}}

\newcommand{\Var}{\operatorname{Var}}
\newcommand{\Cov}{\operatorname{Cov}}

\newcommand{\Po}{\operatorname{Po}}

\newcommand{\Haa}{\ensuremath{{\mathbb H}}}
\newcommand{\Pee}[0]{\ensuremath{{\mathbb P}}}

\newcommand{\Fcal}[0]{\ensuremath{{\mathcal F}}}
\newcommand{\GPo}[0]{\ensuremath{G_{\text{Po}}}}
\definecolor{orange}{RGB}{255,127,0}
\definecolor{pink}{RGB}{255,150,150}

\begin{document}
\title{Hamilton cycles and perfect matchings in the KPKVB model}
\author{Nikolaos Fountoulakis\thanks{School of Mathematics, University of Birmingham, United Kingdom. E-mail: {\small\tt n.fountoulakis@bham.ac.uk}).
Research partially supported by the Alan Turing Institute, grant no. EP/N510129/1.} \and 
Dieter Mitsche\thanks{Institut Camille Jordan, 
Universit\'e Jean Monnet, France. E-mail: {\small\tt dmitsche@unice.fr}.} \and
Tobias M\"uller\thanks{Benoulli Institute, Groningen University, The Netherlands. E-mail: {\small\tt tobias.muller@rug.nl}. Research partially 
supported by NWO grants 639.032.529 and 612.001.409.} \and
Markus Schepers\thanks{Benoulli Institute, Groningen University, The Netherlands. E-mail: {\small\tt m.f.schepers@rug.nl}. Research partially 
supported by NWO grant 639.032.529}}

\maketitle

\begin{abstract}
In this paper we consider the existence of Hamilton cycles and perfect matchings in a random graph model proposed by Krioukov et al.~in 2010. 
In this model, nodes are chosen randomly inside a disk in the hyperbolic plane and two nodes are connected if they are at most a certain 
hyperbolic distance from each other. It has been previously shown that this model has various properties associated with complex networks, including 
a power-law degree distribution, ``short distances'' and a non-vanishing clustering coefficient. 
The model is specified using three parameters: the number of nodes $n$, which we think of as going to infinity, and $\alpha, \nu > 0$, which we think of 
as constant. Roughly speaking $\alpha$ controls the power law exponent of the degree sequence and $\nu$ the average degree.

Here we show that for every $\alpha < 1/2$ and $\nu=\nu(\alpha)$ sufficiently small, the model does not contain a perfect matching 
with high probability, whereas for every $\alpha < 1/2$ and $\nu=\nu(\alpha)$ sufficiently large, the model contains a 
Hamilton cycle with high probability.
\end{abstract}

\section{Introduction} 

A \emph{Hamilton cycle} in a graph is a cycle which contains all vertices of the graph. 
A graph is called \emph{Hamiltonian} if it contains at least one Hamilton cycle. 
A {\em matching} is a set of edges that do not share endpoints and a {\em perfect matching} is a matching that covers all vertices of the graph.

Hamilton cycles and perfect matchings are classical topics in graph theory. 
Historically the existence of Hamilton cycles and perfect matchings in a random graph has been a central theme in the theory of 
random graphs as well. 
In particular, in the theory of the Erd\H{o}s-R\'enyi model the threshold for having a Hamilton cycle as well as the 
simultaneous emergence in the random graph process of a Hamilton cycle together with having minimum degree at least two are among 
the classic results in the field~\cite{ar:Bol84, ar:AjKomSz85, ar:KomSzem83, ar:Posa76, ar:Korshunov77}. 
In the context of random geometric graphs in the Euclidean plane, analogous results have been obtained~\cite{DiazMitschePerez,ar:BBKMW, MPW11}. 
The emergence of Hamilton cycles was also considered in other models, including the preferential attachment model~\cite{ar:AlanXavier} and 
the random $d$-regular graph model~\cite{ar:Nick}.

In this paper, we will consider the problem of the existence of a Hamilton cycle and a perfect matching in a 
model of random graphs that involves points taken randomly in the hyperbolic plane. 
This model was introduced by Krioukov, Papadopoulos, Kitsak, Vahdat and Bogu\~{n}\'a~\cite{ar:Krioukov} in 
2010 - we abbreviate it as \emph{the KPKVB model}. We should however note that the model also goes by several other names in the literature, including
{\em hyperbolic random geometric graphs} and {\em random hyperbolic graphs}.
The model was intended to model complex networks 
and, in particular, it is motivated by the assumption that the properties of complex networks are the expression of a hidden geometry which 
expresses the hierarchies among classes of nodes of the network. 
Krioukov et al.~postulate that this geometry is hyperbolic space.

\subsection*{The KPKVB model}

Given $\nu  \in (0, \infty)$ a fixed constant and a natural number $n > \nu$, we let $R  = 2 \log (n/\nu)$, 
or equivalently  $n = \nu \exp(R /2)$. Also, let $\alpha \in (0, \infty)$.

The hyperbolic plane $\Haa$ is a surface with constant Gaussian curvature $-1$. 
It has several convenient representations (i.e.~coordinate maps), including the Poincar\'e halfplane model, the Poincar\'e
disk model and the Klein disk model. A gentle introduction to hyperbolic geometry and these 
representations of the hyperbolic plane can for instance be found in~\cite{stillwellboek}.
Throughout this paper we will be working with a representation of the hyperbolic plane using {\em hyperbolic polar 
coordinates}. That is, a point $p \in \Haa$ is represented as $(r,\theta)$, where $r$ is 
the hyperbolic distance between $p$ and the origin $O$ and $\theta$ as the angle between the line segment $Op$ and the positive $x$-axis 
(Here, when mentioning ``the origin'' and the angle between the line segment and the positive $x$-axis, we think of $\Haa$ embedded as the 
the Poincar\'e disk in the ordinary euclidean plane.)
We shall denote by $\Dis{R}$ the hyperbolic disk of radius $R$ around the origin $O$, and by $d_{\Haa}(u,v)$ we denote the hyperbolic distance 
between two points $u,v \in \Haa$. 

The vertex set of the KPKVB random graph $G(n;\alpha,\nu)$ consists of~$n$ 
i.i.d.~points in $\Dis{R}$ with probability density function
\begin{equation} \label{eq:pdf}
f_{\alpha,R}(r,\theta) = \frac{\alpha \sinh \alpha r}{2\pi(\cosh \alpha R -1)},
\end{equation}
for $0\leq r <R$ and $0<\theta\leq 2\pi$.

When $\alpha =1$ the distribution  of $(r, \theta)$ given by~\eqref{eq:pdf} is the uniform distribution on $\D$.
For general $\alpha \in (0, \infty)$ Krioukov et al.~\cite{ar:Krioukov} call this the \emph{quasi-uniform} distribution on $\D$. 
In fact, for general $\alpha$ it can be viewed as the uniform distribution on a disk of radius $R$ on the hyperbolic plane with curvature 
$-\alpha^2$.

The KPKVB random graph $G(n;\alpha,\nu)$ is the random graph whose vertex set is a set $V_n$ of $n$ points of chosen i.i.d.~according to the
$(\alpha,R)$-quasi uniform distribution, where any two of them are joined by an edge if they are within hyperbolic distance at most $R$.

Krioukov et al.~\cite{ar:Krioukov} observed that the distribution of the degrees in
$G(n;\alpha,\nu)$ follows a power law with exponent $2\alpha  + 1$, for $\alpha \in (1/2, \infty)$.  
This was verified rigorously by Gugelmann et al.~in~\cite{ar:Gugel}.  
Note that for $ \alpha \in (1/2, 1)$, this quantity is between 2 and 3, which is in line
with numerous observations in networks which arise in applications (see for example~\cite{BarAlb}).
In addition, Krioukov at al.~observed, and Gugelmann et al.~proved rigorously, that 
the (local) clustering coefficient of the graph stays bounded away from zero a.a.s.
Here and in the rest of the paper we use the following notation: If $(E_n)_{n\in \N}$ is a sequence of events then we say 
that $E_n$ occurs \emph{asymptotically almost surely (a.a.s.)}, if $\Pee(E_n) \to 1$ as $n\to \infty$. 

Krioukov et al.~\cite{ar:Krioukov} observed also that the average degree of $G(n;\alpha,\nu)$ is determined via the parameter 
$\nu$ for $\alpha \in (1/2, \infty)$. 
This was rigorously verified in~\cite{ar:Gugel} too. In particular, they proved that
the average degree tends to  $2\alpha^2 \nu / \pi(\alpha-\frac12)^2$ in probability.

%

In~\cite{BFMgiantEJC}, it was established
that $\alpha = 1$ is the critical point for the emergence of a giant component in $G(n;\alpha,\nu)$. In particular, when $\alpha \in (0,1)$, the 
fraction of the vertices contained in the largest component is bounded away from 0 a.a.s., whereas if $\alpha \in (1, \infty)$, the largest
component is sublinear in $n$ a.a.s.
For $\alpha = 1$, the component structure depends on $\nu$. If $\nu$ is large enough,
then a giant component exists a.a.s., but if $\nu$ is small enough, then a.a.s. all components are sublinear~\cite{BFMgiantEJC}. 

 In~\cite{ar:FounMull2017} this picture is sharpened. There, the first and the third author showed that the fraction of vertices 
 belonging to the largest component converges in probability to a constant which
depends on $\alpha$ and $\nu$.  For $\alpha =1$, the existence of a critical value $\nu_0 \in (0, \infty)$ is established such that 
when $\nu$ crosses $\nu_0$ a giant component emerges a.a.s.~\cite{ar:FounMull2017}. 
In~\cite{KiwiMit} and ~\cite{KiwiMit2017+}, the second author together with Kiwi considered the size of the second largest component 
and showed that when $\alpha \in (1/2, 1)$, a.a.s., the second largest component has polylogarithmic order with exponent 
$1/(\alpha -1/2)$.

Apart from the degree sequence, clustering and component sizes, the graph distances in this model have also been considered.
In~\cite{KiwiMit} and~\cite{ar:FriedKrohmerDiam} a.a.s.~polylogarithmic upper and lower bounds on the diameter of the largest component 
are shown, and in~\cite{ar:MullerDiam}, these were sharpened to show that $\log n$ is the correct order of the diameter.
Furthermore, in~\cite{ar:AbdBodeFound} it is shown that for $\alpha \in (1/2,1)$ the largest component is what is called an 
\emph{ultra-small world}: it exhibits doubly logarithmic typical distances. 

Results on the global clustering coefficient were obtained in~\cite{CF16}, and on the evolution of graphs on more general
spaces with negative curvature in~\cite{F12}. 
The spectral gap of the Laplacian of this model was studied in~\cite{KM18}.

The first and third author together with Bode~\cite{BFMconnRSA}, showed that $\alpha=1/2$ is the critical 
value for connectivity: that is, when $\alpha \in (0, 1/2)$, then $G(n;\alpha,\nu)$ is a.a.s.~connected, whereas $G(n;\alpha,\nu)$ is 
a.a.s.~disconnected when $\alpha \in (1/2, \infty)$. 
The second half of this statement is in fact already immediate from the results of Gugelmann et al.~\cite{Gugel} : there it is shown that 
for $\alpha>1/2$, a.a.s., there are linearly many isolated vertices.
For $\alpha =1/2$, the probability of connectivity tends to a limiting value that is function of $\nu$ that is continous and non-decreasing and 
that equals one if and only if $\nu \geq \pi$.

\subsection*{Our results}

In the present paper, we explore the existence of Hamilton cycles and perfect matchings in the 
$G(n;\alpha,\nu)$ model. In the light of the result on isolated vertices mentioned above, 
the question is non-trivial only for $\alpha \leq 1/2$.
A perfect matching trivially cannot exist when $n$ is odd. For this reason we find it convenient
to switch to considering {\em near perfect matchings} from now on. That is, matchings that cover all but at most one vertex.
(So if $n$ is even a near perfect matching is the same as a perfect matching; and the existence of a Hamilton cycle
implies the existence of a near perfect matching.) 

Our main results shows that in the regime $\alpha < 1/2$ regime, a.a.s., the existence of a Hamilton cycle as well as of a 
(near) perfect matching has a non-trivial phase transition in $\nu$: 

\begin{thm}\label{thm:main}
For all positive real $\alpha < \frac{1}{2}$, there are constants $\nu_0=\nu_0(\alpha)$ and $\nu_1 = \nu_1 (\alpha)$
such that the following hold.
For all $0<\nu < \nu_0$, the random graph $G(n;\alpha,\nu)$ a.a.s. does not have a near perfect matching (and consequently no Hamilton cycle either). 
For all $\nu > \nu_1$, $G(n;\alpha,\nu)$ a.a.s. has a Hamilton cycle. 
\end{thm}

To our knowledge, this is the first time this problem is considered for the $G(n;\alpha, \nu)$ model. 
We conjecture that the dependence on $\nu$ is sharp. 

\begin{conj}
For every $0<\alpha<1/2$ there exists a critical $\nu_c = \nu_c(\alpha) > 0$ such that
when $\nu < \nu_c$ a.a.s. $G(n;\alpha,\nu)$ has no near perfect matching, whereas if $\nu > \nu_c$ then
a.a.s. $G(n;\alpha,\nu)$ has a Hamilton cycle.
\end{conj}

A natural question to ask is what happens in the case $\alpha=1/2$.
Does there exist $\nu$ large enough so that the graph a.a.s. becomes Hamiltonian in this case as well?

It would also be interesting to explore the relation of Hamiltonicity with the property of 2-connectivity.
If the above conjecture is true, is there a similar behaviour for the property of 2-connectivity? 
If yes, are the corresponding critical constants $\nu_c$ equal? 

\vspace{1ex}

\textbf{Outline of proof.} 
The proof of Theorem~\ref{thm:main} has two parts: in a nutshell, in the first part we show that for $\nu$ small enough, the number of 
vertices close to the boundary of the disk of radius $R$ having no neighbor close to the boundary of the disk will be bigger than the 
total number of vertices relatively close to the centre of the disk. 
Hence, the former vertices would have to be all matched to distinct vertices close to the centre of the disk, but that cannot happen. 
For the second part, we show that for $\nu$ large enough, we can tessellate the disk in such a way, so that iteratively, from the boundary 
towards the center of the disk, we can maintain a set of vertex-disjoint cycles and isolated points, which will eventually be merged 
close to the centre. The fact that $\nu$ is large enough makes the density of vertices in each cell of the tessellation high enough 
so that this procedure terminates successfully.

 \section{Preliminaries}

\subsection{Probabilistic preliminaries} \label{sec:Pois}
To prove Theorem~\ref{thm:main}, we will perform our analysis in the \emph{poissonisation} of the
$G(n;\alpha, \nu)$ model. There, the vertex set is the point set of a Poisson point process on $\D$ with $n$ points
on average. Although the independence that accompanies the Poisson point process facilitates our proofs, when doing standard de-poissonisation, we need to show a slightly stronger version of Theorem~\ref{thm:main}. We give details here.

We denote the Poissonized version of the KPKVB model by $\GPo(n;\alpha,\nu)$. 
The vertex set of this random graph is the set of points of the Poisson point 
process $\PP$ on $\D$ with intensity $n \cdot  \frac{1}{2\pi} f_{\alpha,R}$. 
The set of edges of $\GPo (n;\alpha,\nu)$ consists of those pairs of points of $\PP$ which are at hyperbolic distance at most $R$.
Alternatively, the Poissonized KPKVB model $\GPo(n;\alpha,\nu)$ can be constructed as follows. 
Consider an infinite supply of i.i.d.~points $p_1, p_2,\dots$, chosen according to the
$(\alpha,R)$-quasi uniform distribution, and a Poisson random variable $Z \eqd \Po(n)$. The vertex set 
of $\GPo$ is now the set of points $p_1,\dots, p_Z$ and again we add edges between pairs at hyperbolic distance at most $R$.

The function $n \cdot  \frac{1}{2\pi} f_{\alpha,R}$ is the {\em intensity measure} associated with  $\PP$.
This means in particular that for any Borel subset  $A \subseteq \D$ the expected number of 
points that fall in $A$ equals 

$$ \im (A) = n  \cdot \frac{1}{2\pi} \int_A f_{\alpha,R} (r, \theta) dr d\theta. $$

\noindent
We set $\PP (A) := \PP \cap A$; hence $|\PP(A)| \eqd \Po (\im (A))$. 
%
%
An elementary, but key, observation is that conditional on $|\PP|=n$, the process $\PP$ is distributed as $V_n$. 
In other words, the probability space of the process $V_n$ can be realised as the 
space of $\PP$ conditional on $|\PP|=n$.

The following observation is well known. We include its proof here for completeness. 

\begin{lemma}\label{lem:Poisson}
Let $\Fcal$ be a graph property (formally a family of graphs closed under isomorphism). 
We then have that 
\begin{equation*}
\Pee ( \GPo(n;\alpha,\nu) \in \Fcal \mid |\PP|=n) =1-o(1), \ \mbox{if} \ \Pee ( \GPo(n;\alpha,\nu) \in \Fcal ) = 1 -o(n^{-1/2}). 
\end{equation*}
Thus, if $\Pee(\GPo(n;\alpha,\nu) \not \in \Fcal) = o(n^{-1/2})$, then $\Pee(G(n;\alpha,\nu) \not \in \Fcal) 
= o(1)$.
\end{lemma}
\begin{proof}
By Stirling's formula
\[ \pr(|\PP|=n) = \frac{n^n}{n!}e^{-n} = (1+o(1))\frac{n^n}{\sqrt{2\pi}n^{n+\frac{1}{2}}e^{-n}}e^{-n} = (1+o(1))\frac{1}{\sqrt{2\pi n}}\]
So as $n \rightarrow \infty$, writing $E_n := \{ \GPo(n;\alpha,\nu) \in \Fcal \}$:

$$\Pee (E_n^c) \geq \Pee (E_n^c \mid |\PP|=n) \cdot \Theta \left(n^{-1/2}\right). $$

Therefore, if $ \Pee (E_n^c)  =o(n^{-1/2})$, we deduce that $\Pee (E_n^c \mid |\PP|=n) =o(1)$.
\end{proof}

We then apply a standard depoissonisation technique: using Lemma~\ref{lem:Poisson}, we will develop our arguments in the poissonisation of 
the KPKVB model.
More precisely, we will show that $\GPo(n;\alpha, \nu)$ satisfies certain events with sufficiently high probability 
(that is, with probability at least $1-o(n^{-1/2})$, and we then use Lemma~\ref{lem:Poisson} to deduce that $G(n;\alpha,\nu)$ 
also satisfies them a.a.s., that is, with probability $1-o(1)$.

Another useful tool that will allow us to compute expectations of sums over the points of $\PP$ is the \emph{Campbell-Mecke} formula. 
Let $\mathcal{P}$ be a point process on a metric space $\mathcal{S}$ with density $\rho$. 
Let $\mathcal{N} = \mathcal{N} (\mathcal{S})$ be the set of all countable point configurations in $\mathcal{S}$
equipped with the $\sigma$-algebra of the point process (that is, for any open subset 
$A\subseteq \mathcal{S}$ and any non-negative integer $m$ define a basic measurable subset of 
$\mathcal{N}$ which consists of all configurations which have exactly $m$ points in $A$). 
Now, let $h : \mathbb{R}^{k} \times \mathcal{N} \rightarrow \mathbb{R}$ be a measurable function. 
The Palm theory of Poisson point processes on metric spaces~\cite{bk:LastPenrose} yields: 

\begin{equation} \label{eq:Mecke}
\E \left( \sum_{p_1, \ldots, p_k \in {\mathcal P}, \atop \text{distinct}} h(p_1,\ldots, p_k, \mathcal{P})  \right) =
\int_\mathcal{S} \cdots \int_\mathcal{S} \E (h(x_1,\ldots, x_k, \mathcal{P})) d\rho (x_1) \cdots d \rho (x_k),
\end{equation} 

where the sum ranges over all those $k$-tuples of points which contain no repetitions. 
We also use the following form of Chernoff's bound:

\begin{lemma}\label{lem:Chernoff} 
For a Poisson random variable $X$ with expectation $\mu$, and a positive integer $k \leq \mu$, 
\begin{align*}
\pr(X \leq k) \leq e^{-\mu H(\frac{k}{\mu})},
\end{align*}
 where $H(a) = 1-a+a\ln a$ for $a>0$; in particular if $\frac{k}{\mu} \leq \frac{1}{2}$ it holds that $H(\frac{k}{\mu}) \geq \frac{1}{2}(1-\ln 2)>0$. 
 \end{lemma}
 
 A proof can for instance be found in~\cite{Penroseboek}.
 
\subsection{Geometric preliminaries}

For two points $p_1=(r_1,\theta_1)$ and $p_2=(r_2,\theta_2)$ 
in $\D$, we let $\theta (p_1,p_2)$ be their \emph{angular distance} which we define as: 
$$\theta (p_1,p_2) = \min \{|\theta_1 - \theta_2|, 2\pi - |\theta_1- \theta_2| \}.$$
Note that $\theta (p_1,p_2) \in [0,\pi]$. 

The hyperbolic law of cosines relates the angular distance $\theta (p_1,p_2)$ between two points $p_1,p_2$ 
with their hyperbolic distance:
\begin{equation} \label{eq:cosines_law}
\cosh (d_{\Haa}(p_1,p_2))  =
 \cosh (r(p_1) ) \cosh (r(p_2))  - \sinh (r(p_1)) \sinh (r(p_2))
\cos ( \theta (p_1,p_2) ).
\end{equation}

Now, for $r_1,r_2 \in [0,R)$, we let $\theta_R (r_1,r_2) \in [0,\pi]$ be such that 
$$\cosh (R)  = \cosh (r_1 ) \cosh (r_2)  - \sinh (r_1) \sinh (r_2)
\cos ( \theta_R (r_1,r_2) ). $$
For two points $p_1, p_2$ with $r(p_1) = r_1$ and $r(p_2)=r_2$, if 
$\theta (p_1,p_2) = \theta_R (r_1,r_2)$, then $d_{\Haa}(p_1,p_2 )  = R$. 
Equation~\eqref{eq:cosines_law} implies that $d_{\Haa}(p_1,p_2) \leq R$ if and only if 
$\theta (p_1,p_2) \leq \theta_R(r(p_1), r(p_2))$. 

The ball $B(p)$ of radius $R$ around $p$ inside $\D$ is thus defined as 
$$ B(p)  := \{ p' \in \D \ : \ \theta(p,p') \leq \theta_R (r(p),r(p')) \}.$$
If $p$ is a vertex either in the vertex set of $G(n;\alpha,\nu)$ or in the vertex 
set of $\GPo (n;\alpha,\nu)$, then it is adjacent precisely to any other vertex that 
belongs to $B(p)$. 

It will be convenient for our analysis to express $\theta_R(r_1,r_2)$ explicitly as a function of $r_1,r_2$. 
We will make use of the following Lemma from~\cite{ar:FounMull2017}, which does this. 
\begin{lemma}[\cite{ar:FounMull2017}, Lemma 28]\label{lem:adjang}
There exists a constant $K>0$ such that for every $\epsilon >0$ and $R$ sufficiently large, the following holds: for every $r_1, r_2 \in [\epsilon R, R]$ with $r_1+r_2>R$, we have 
$$2e^{\frac{1}{2}(R-r_1-r_2)} - K e^{\frac{3}{2}(R-r_1-r_2)} \leq \theta_R (r_1,r_2) \leq 2e^{\frac{1}{2}(R-r_1-r_2)} + Ke^{\frac{3}{2}(R-r_1-r_2)}.$$
Moreover, if $r_1, r_2 < R-K$, then $\theta_R (r_1,r_2) \geq 2e^{\frac{R-r_1-r_2}{2}}$.
\end{lemma}

\section{Non-existence of perfect matching for sufficiently small $\nu$}


The following theorem yields the first part of Theorem~\ref{thm:main}.

\begin{thm}
For all positive real $\alpha < \frac{1}{2}$, there is a $\nu_0=\nu_0(\alpha)>0$ such that for all $0<\nu < \nu_0$, the random 
graph $\GPo(n;\alpha,\nu)$ does not have a near perfect matching w.p. $1- o(n^{-1/2})$.
\end{thm}

\begin{proof}
The strategy is as follows. Let $s=\frac{1}{\alpha}>2$.
Let $N_s$ be number of vertices with radial coordinate at least $R-s$ and with no neighbour with radial coordinate at least $R-s$.
Let $M_s$ be the number of vertices with radial coordinate at most $R-s$. Hence, $M_s$ is the number of points 
of $\PP$ inside the disk of radius $R-s$ and $N_s$ is a subset of the annulus 
$\An{s} = \Dis{R} \setminus \Dis{R-s}$ of width $s$. 
If there is a perfect matching, then $M_s \geq N_s$ because a vertex with no neighbour with radius at least 
$R-s$ must be matched to a vertex with radius less than $R-s$, so distinct vertices counted by $N_s$ must be matched to distinct vertices counted by $M_s$. 
If it is shown that $M_s$ and $N_s$ are concentrated around their expectation w.p. $1-o(n^{-1/2})$ and that 
$\E M_s =(1+o(1)) c_{M_s} n$ 
and $\E N_s =(1+o(1)) c_{N_s} n$ as $n \rightarrow \infty$ and $c_{M_s} < c_{N_s}$, then there will be no near perfect matching and hence no Hamilton cycle w.p. $1-o(n^{-1/2})$.

We observe that $M_s \stackrel{\Delta}{=} \Po (\mu_{n,\alpha,\nu}(\Dis{R-s}))$, where
\begin{align*}
\mu_{n,\alpha, \nu} (\Dis{R-s}) = n \cdot   \frac{ \cosh \alpha (R-s) -1}{\cosh \alpha R -1}  
\cdot \frac{1}{2\pi}\int_{0}^{2\pi} d\theta
\sim n \cdot  \frac{\frac{1}{2}e^{\alpha(R-s)}}{\frac{1}{2}e^{\alpha R} }
 =   n e^{-\alpha s} =  n  e^{-1}.
\end{align*}

As $M_s$ is Poisson distributed, we deduce that  $\E M_s = \Var M_s \sim ne^{-1}$. (Here and elsewhere 
we write $a_n \sim b_n$ to denote that $a_n/b_n=1+o(1)$.)
By Chebyshev's inequality, it follows that for all 
$\epsilon>0$
\begin{align*}
\Pee(|M_s - \E M_s| \geq \epsilon \E M_s) \leq \frac{\Var(M_s)}{\epsilon^2 (\E M_s)^2} = O(n^{-1}) = o(n^{-1/2}).
\end{align*}

Our aim now is to give a lower bound on $\E N_s$. To this end, we will show that $\Var(N_s) / \E^2 N_s=
 o(n^{-1/2})$. 

For a point $p \in \D$ we let $\lambda (p, \PP)$ denote the indicator random variable which is equal to 1 if and 
only if $\PP (\check{B}_s(p)):=\PP ( (B(p) \setminus \{p\} ) \cap \An{s}) = \emptyset$. 
In other words, $\lambda (p,\PP)$ is equal to 1 if and only if no point of 
$\PP \setminus \{p\}$ is contained in $B(p) \cap \An{s}$. 

We can write 
$$N_s = \sum_{p \in \PP \cap \An{s}} \lambda (p,\PP).$$
The Campbell-Mecke formula~\eqref{eq:Mecke} will allow us to calculate the expected value of $N_s$:
\begin{align}\label{eq:expectation_Ns}
\E N_s &= n\cdot  \frac{1}{2\pi} \int_{\An{s}}\E \lambda ((r,\theta), \PP ) f_{n, \alpha, \nu} (r, \theta) dr d\theta \nn \\
&= n \cdot  \frac{1}{2\pi}  \int_{R-s}^R \int_0^{2\pi} \E \lambda ((r,\theta), \PP ) f_{n, \alpha, \nu} (r, \theta) dr 
d\theta \nn \\
&= n \cdot  \int_{R-s}^R  \E \lambda ((r,0), \PP ) f_{n, \alpha, \nu} (r, \theta) dr,
\end{align}
where the first equality holds since $\lambda((r,\theta),\PP)=0$, if  and only if $r < R -s$ and the last one since 
$\E \lambda ((r,\theta), \PP )$ is invariant with respect to $\theta$.

We have 

$$\check{B}_s ((r,0)) =\{(r',\theta') \in \R^2: R-s \leq r' < R, |\theta'| \leq \theta_R (r,r') \},$$ 

and, therefore, 
\begin{align*}
\im(\check{B}_s ((r,0))) =n \cdot \int_{R-s}^R \frac{2\theta_R (r,r')}{2\pi} \frac{\alpha \sinh \alpha r'}{\cosh \alpha R -1}dr'.
\end{align*}
We can give an asymptotic approximation to this integrand. 
From Lemma~\ref{lem:adjang}, we infer that for $n$ large enough, uniformly over all $r,r' \geq R-s$:
\begin{align*}
\theta_R (r,r') = 2 e^{\frac{R-r-r'}{2}}(1+O(e^{R-r-r'})) = 2 e^{\frac{R-r-r'}{2}}(1+O(e^{-R}) ),
\end{align*}
and 
\begin{align*} 
\frac{\sinh \alpha r'}{\cosh \alpha R -1} \sim e^{-\alpha (R-r')}.
\end{align*}
Therefore, 
\begin{align*}
\im (\check{B}_s ((r,0))) & \sim n \cdot e^{-r/2} \frac{\alpha}{\pi} \int_{R-s}^R  e^{(1/2-\alpha) (R-r')} dr' 
= n \cdot e^{-r/2} \frac{\alpha}{\pi} \int_{0}^s  e^{(1/2-\alpha) y} dy  \\
&\stackrel{n= \nu e^{R/2}}{=} \nu \cdot e^{(R-r)/2} \frac{\alpha}{\pi (1/2 -\alpha)} \left( e^{(1/2-\alpha)s} - 1\right). 
\end{align*}
But $R-r \leq s$.
We use that $1/s = \alpha$, and set $$e^{s/2}\frac{\nu \alpha}{\pi (1/2 -\alpha)} \left( e^{1/(2\alpha)-1} - 1\right) =: 
\nu c_{\alpha}.$$
Thus, for $n$ sufficiently enough $\im (\check{B}_s(p)) \leq \nu c_\alpha$.
So for any such $n$ 
$$ \E \lambda ((r,0), \PP) = \mathbb{P} (\Po (\im (\check{B}_s (p))) = 0) \geq e^{-\nu c_\alpha}.$$

If we substitute this into~\eqref{eq:expectation_Ns}, we get the following lower bound: 
\begin{align*}
\E N_s &\geq n e^{-\nu c_\alpha} \cdot  \int_{R-s}^R  f_{n, \alpha, \nu} (r, \theta) dr \\
&= n e^{-\nu c_\alpha}  \cdot \frac{\cosh (\alpha (R))- \cosh (\alpha (R-s)))}{\cosh (\alpha R)-1} \sim 
n e^{-\nu c_\alpha} (1- e^{-1}).
\end{align*}
As $e> 2$, we have that $e^{-1} < 1 - e^{-1}$. Select $\nu > 0$ sufficiently small so that 
$$ e^{-1} < e^{-\nu c_\alpha} (1- e^{-1}).$$
Thereafter, we choose $\epsilon>0$ sufficiently small so that 
$$ e^{-1} (1+\epsilon) < (1-\epsilon) e^{-\nu c_\alpha} (1- e^{-1}). $$

We will use Chebyshev's inequality to show that w.p. $1-o(n^{-1/2})$
$$N_s \geq (1-\epsilon) \E N_s. $$ 
Since $M_s < n e^{-1} (1+\epsilon)$ w.p. $1-o(n^{-1/2})$, the union bound implies 
 $M_s < N_s$ w.p. $1-o(n^{-1/2})$. 

We will show that $\Var (N_s) = O(n)$. 
To bound the variance, let us set 
$$\Cov (p_1,p_2) := \E(\lambda(p_1,\PP) \lambda (p_2, \PP)) - \E (\lambda (p_1, \PP)) \E (\lambda (p_2, \PP)).$$ 
We write 

\begin{align*}
\Var (N_s) = \E \left( \sum_{p_1,p_2 \in \PP,\atop \text{distinct}} \Cov (p_1, p_2 )\right) 
+\E \left( \sum_{p \in \PP} \Cov (p,p)\right).
\end{align*}

We will use the Campbell-Mecke formula~\eqref{eq:Mecke} to calculate these sums. 
For the former one, we have:

\begin{align*}
\E & \left( \sum_{p_1,p_2 \in \PP,\atop \text{distinct}} \Cov (p_1, p_2 )\right)  \\
&=  n^2 \cdot \left( \frac{1}{2\pi} \right)^2 \int_{R-s}^R \int_0^{2\pi}\int_{R-s}^R \int_0^{2\pi}   
\E (\Cov ((r_1,\theta_1),(r_2,\theta_2))) f_{n,\alpha, \nu} (r_1) f_{n,\alpha, \nu} (r_2)
dr_1 d\theta_1 dr_2 d\theta_2.
\end{align*}

Let $p_1, p_2  \in \An{s}$. We claim that for $n$ sufficiently large,
$\check{B}_s (p_1)  \cap \check{B}_s (p_2) = \emptyset$, if 
$\theta (p_1,p_2) > 5\nu e^{s}/n$.
Indeed, Lemma~\ref{lem:adjang} implies that for any $r, r' > R-s$, we have 
$\theta_R (r,r') < 2 e^{\frac12 (R-r-r')} (1 + O(e^{\frac32(R-r-r')})) < 2 e^{s - R/2} (1+O(e^{-\frac32 R})) < 2.5 e^{s-R/2}$, 
for $n$ sufficiently large.  
Thus, for any such $n$, if $p \in \check{B}_s (p_1)  \cap \check{B}_s (p_2)$, then 
$\theta (p,p_1), \theta (p,p_2) < 2.5 e^{s-R/2}$. This would imply that $\theta (p_1,p_2) < 5 e^{s-R/2}$. 
So if $\theta (p_1,p_2) > 5 e^{s-R/2}$, then $\check{B}_s (p_1)  \cap \check{B}_s (p_2) = \emptyset$.

This implies that for $n$ sufficiently large, whenever $\theta (p_1,p_2) > 5\nu e^{s}/n$, we have 
$\E (\Cov (p_1,p_2))=0$. Moreover, $| \E (\Cov (p_1,p_2)) | \leq 1$. 
So, 

\begin{align*}
\E  \left( \sum_{p_1,p_2 \in \PP,\atop \text{distinct}} \Cov (p_1, p_2 )\right)  = O(n).
\end{align*}

Regarding the second term, we use that $\E \Cov (p,p) \leq \E \lambda (p,\PP)$ and bound 
$$ \E \left( \sum_{p \in \PP} \Cov (p,p)\right) \leq \E \left( \sum_{p \in \PP} \E \lambda (p,\PP) \right) = \E (N_s) = O(n).$$
These two imply that 
$$ \Var (N_s) = O(n).$$ 
 Chebyshev's inequality yields
$$\Pee (N_s \geq (1-\epsilon) \E N_s ) \leq \frac{\Var (N_s)}{\epsilon^2 \E^2 (N_s)} = O\left(\frac1n \right)
=o(n^{-1/2}). $$
\end{proof}

\section{Existence of Hamilton cycles for  sufficiently large $\nu$}

The aim of this section is to prove the existence of a Hamilton cycle in $\GPo (n;\alpha,\nu)$ when $\nu$ is large enough with sufficiently 
high probability. 
\begin{thm}\label{thm:HC}
For all positive real $\alpha < \frac{1}{2}$, there is a $\nu_1=\nu_1(\alpha)$ such that for all $\nu >\nu_1$, the random graph $\GPo(n;\alpha,\nu)$ has a Hamilton cycle and hence also a near perfect matching with 
probability $1-o(n^{-1/2})$.
\end{thm}

\subsection{A useful tiling}

We consider the following tiling

\begin{align*}
T_{i,j} =\{ (r,\theta) \in \mathcal{D}_R: R-(i+1)2 \ln 2 \leq r <R- i2\ln 2, j \frac{2\pi}{n_{i}}< \theta \leq (j+1) \frac{2\pi}{n_{i}} \}
\end{align*} 

where $n_{i} =n_{i,R}= 2^{4-i+\lfloor \frac{R}{2\ln 2} \rfloor} \in \N_0:=\N \cup \{ 0\}$, for $i \in \N_0$, $i \leq i_{max}=\lceil \frac{0.9 R}{2\ln 2} \rceil$ and $j \in \N_0$, $j < n_{i}$. We call $i,j$ admissible if they satisfy these constraints. Note that for all admissible $i$, the parameter $n_{i}$ is an integer and, in fact, a power of $2$, as the exponent $4-i+\lfloor \frac{R}{2\ln 2} \rfloor$ is an integer. Moreover, for $i \leq i_{max}$ the exponent is also at least $4-\frac{0.9 R}{2\ln 2}+\frac{R}{2\ln 2} -1= 3 + 0.1 \frac{R}{2\ln 2}>0$.

We call the collection of tiles with a fixed given $i$ the \emph{$i$-th layer}. These are the tiles in the $i$-th annulus where we start counting from zero at the boundary of the disk. Note that there are $n_{i}$ tiles in the $i$-th layer and the tiling covers the annulus with exterior radius $R$ and interior radius $R-i_{max} 2\ln 2 = (1+o(1))0.1 R$ (in particular, the most interior layer $i_{max}$ is contained in the smaller disk with radius $\frac{R}{2}$ around the origin). A schematic picture is shown in Figure~\ref{fig:diskTiling}.


\begin{figure}
\includegraphics[scale=0.6]{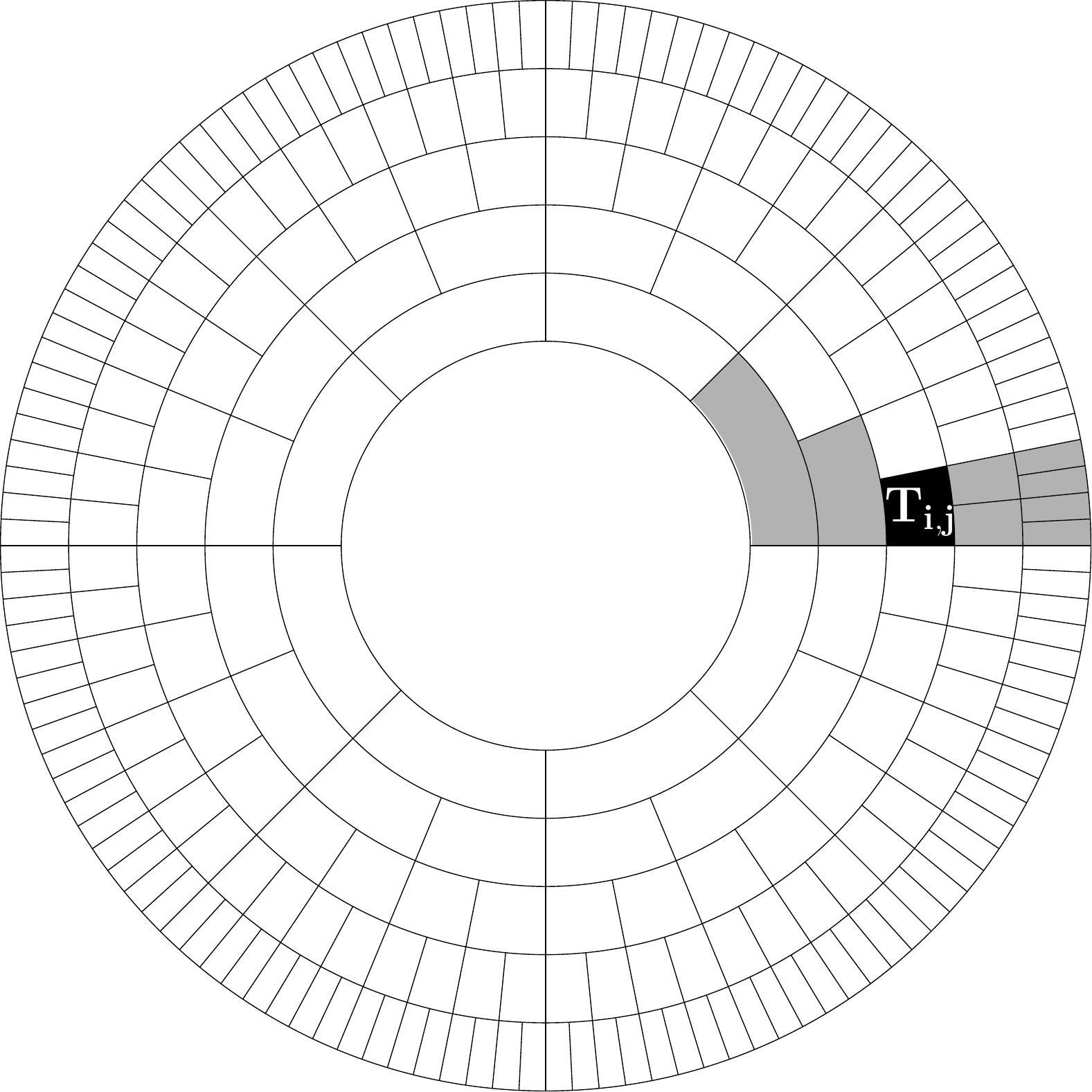}
\caption{(Partial) tiling in the hyperbolic disk; example of a tile $T_{i,j}$ (coloured black) and the tiles which are guaranteed to lie within its neighbourhood ball by Lemma~\ref{lem:adjtiles} (coloured black and grey).
}\label{fig:diskTiling}
\end{figure}

We say that a tile $T_{i',j'}$ is \emph{below} the tile $T_{i,j}$, if $i'\leq i$ and the sector defined by $T_{i,j}$ 
contains $T_{i',j'}$. 
\begin{lemma}[Adjacency among the tiles]\label{lem:adjtiles}
For admissible indices $i,j$, any point $p \in T_{i,j}$ is within distance $R$ from any point $p'$ in any tile below the tile $T_{i,j}$.
\end{lemma}
\begin{proof}
Let $p=(r,\theta) \in T_{i,j}$ and $p'=(r',\theta') \in T_{i',j'}$ be a vertex in any tile below tile $T_{i,j}$ (in the sense of the statement above). 
Note that $r' \ge r$ must hold.
Then, the angular distance $|\theta-\theta'|_{2\pi}$ between $p$ and $p'$ is at most the angular width of the tile $T_{i,j}$ which is 
$$\frac{2\pi}{n_{i}} = 2^{-3}\pi 2^{i-\lfloor \frac{R}{2\ln 2} \rfloor} \leq 2^{i-1} e^{-\frac{R}{2}}.$$

On the other hand, we know that the radial coordinates satisfy $r \leq R- i 2\ln 2$ and $r' <R$. If $r+r'\leq R$, we have adjacency by the triangle inequality. If $r+r' > R$ and using that $r,r' \geq (1+o(1))0.1 R$ (as remarked earlier), we distinguish two cases:
\begin{enumerate}
\item If $r,r' <R-K$ (with $K$ as in Lemma~\ref{lem:adjang}), then by the last part of Lemma~\ref{lem:adjang}, it holds that
\begin{align*}
\theta_R (r,r') \geq 2 e^{\frac{R-r-r'}{2}} \geq 2 e^{\frac{i2\ln 2 - R}{2}} = 2^{i+1} e^{-\frac{R}{2}}.
\end{align*}
\item Otherwise we may assume that $r' \ge R-K$ holds, while still $r,r' \geq (1+o(1))0.1R$. 
Therefore, $R-r-r' \leq R-(R-K)-(1+o(1))0.1R = -(1+o(1))0.1R +K$, hence the error term in 
Lemma~\ref{lem:adjang} is $K e^{\frac{3}{2}(R-r-r')} = o(e^{\frac{1}{2}(R-r-r')})$ and it follows that

\begin{align*}
\theta_R(r,r') \geq 2e^{\frac{R-r-r'}{2}} - o(e^{\frac{1}{2}(R-r-r')}) > e^{\frac{R-r-r'}{2}} \geq e^{\frac{i2\ln 2-R}{2}} = 2^i e^{-\frac{R}{2}}.
\end{align*}

\end{enumerate}

We conclude that $|\theta-\theta'|_{2\pi} \leq \theta_R(r,r')$ from which the claim follows.
\end{proof}

We will denote by $N(T_{i,j})$ the number of points falling into $T_{i,j}$.

\begin{lemma}[Expected number of points in a tile]\label{lem:exppoints}
Let $\alpha, \nu >0$. For admissible indices $i,j$, the expected number of points falling into $T_{i,j}$ satisfies: 
$$\E N(T_{i,j}) =\Theta(\nu 2^{i(1-2\alpha)} ).$$
\end{lemma}

\begin{proof}
The expected number of points falling into $T_{i,j}$ is given by:
\begin{align*}
\im (T_{i,j}) &= n\cdot \int_{R-(i+1) 2\ln 2}^{R-i2\ln 2} \int_{j \frac{2\pi}{n_{i}}}^{(j+1) \frac{2\pi}{n_{i}}} \frac{\alpha \sinh \alpha r}{2\pi(\cosh \alpha R -1)} d\theta dr \\
&=n \cdot \int_{R-(i+1)2\ln 2}^{R-i2\ln 2} \frac{\alpha \sinh \alpha r}{n_{i}(\cosh \alpha R -1)}dr \\
&= n \cdot \frac{\cosh( \alpha (R-i2\ln2)) - \cosh (\alpha (R-(i+1)2\ln 2))}{n_{i}(\cosh \alpha R -1)}.
\end{align*} 
As $i \leq i_{max}$, we have that $R-i2\ln 2 \geq 0.1 R \rightarrow \infty$, and hence
$$\cosh( \alpha (R-i2\ln 2)) =(1+o(1)) \frac{1}{2}e^{\alpha (R-i2\ln 2)},$$
 $$\cosh( \alpha (R-(i+1)2\ln 2)) =(1+o(1))\frac{1}{2}e^{\alpha (R-(i+1)2 \ln 2)}$$ and 
 $$\cosh \alpha R =(1+o(1)) \frac{1}{2}e^{\alpha R}.$$ Furthermore, $n_{i} = 2^{4-i+\lfloor \frac{R}{2\ln 2}\rfloor} =\Theta( 2^{-i+\frac{R}{2\ln 2}})=\Theta(2^{-i}e^{\frac{R}{2}})$. We conclude:
\begin{align*}
\E N(T_{i,j}) &=\Theta\left(n\frac{e^{\alpha (R-i2\ln 2)} - e^{\alpha (R-(i+1)2\ln 2)}}{ 2^{-i}e^{\frac{R}{2}} e^{\alpha R}}\right) = \Theta(n 2^{i} e^{-\frac{R}{2}} e^{-i2\alpha\ln 2} (1-e^{-2\alpha \ln 2})) \\
&= \Theta(n 2^{i(1-2\alpha)} e^{-\frac{R}{2}})
\end{align*}
Finally, using that $R=2\ln \frac{n}{\nu}$, that is, $n= \nu e^{\frac{R}{2}}$, yields the claim.
\end{proof}

\subsection{A procedure for finding a Hamilton cycle}

In this subsection we describe the strategy of our procedure for finding a Hamilton cycle in a graph which is embedded in the hyperbolic disk $\mathcal{D}_R$ and which 
makes use of the tiling $(T_{i,j})_{i,j \in \N_0, i \leq i_{max}, j<n_{i}}$ defined above.
Roughly speaking, the procedure iterates through the layers of the tiling, working upwards from the $0$-th layer to layer $i_{\max}$, 
gathering a suitable collection of vertex-disjoint cycles and isolated vertices. 
When processing the tile $T_{i,j}$, it merges as many vertex-disjoint cycles and isolated points from previous iterations that are below the tile as possible. 
Once the procedure has reached the maximum layer which is completely contained in the smaller disk with radius $\frac{R}{2}$, the procedure 
attempts to merge all the remaining cycles and points. 

We now describe the procedure in more detail.
For each tile $T_{i,j}$ we will define a random variable $D_{i,j}$ called demand, which will be used later in the 
probabilistic analysis to show that the procedure terminates successfully. Recall that $N(T_{i,j})$ denotes the number 
of points in tile $T_{i,j}$ and note that the collection of $N(T_{i,j})$ for admissible $i,j$ are independent Poisson random variables for the poissonised KPKVB model.

\begin{lemma}[Cycle merging, see Figure~\ref{fig:cyclemerge}] \label{lem:cyclemerging}
If the vertices strictly below tile $T_{i,j}$ can be covered by $x$ vertex-disjoint cycles and isolated vertices and the number of vertices 
in tile $T_{i,j}$, $y=N(T_{i,j}) \geq 3$, then the set of all vertices below $T_{i,j}$ (including those in $T_{i,j}$) 
can be covered by $\max \{1,x-y+1\}$ cycles and points.

Furthermore, if additionally $y>x$, then the vertices below $T_{i,j}$ can be covered by a single cycle which has $y-x$ edges within $T_{i,j}$.
\end{lemma}

\begin{proof}
If $y=N(T_{i,j}) \geq 3$, then the vertices in $T_{i,j}$ form a cycle by Lemma~\ref{lem:adjtiles}. 
Each of its $y$ edges can be used to merge this cycle in $T_{i,j}$ with a cycle or point strictly below $T_{i,j}$: 
to pick up a cycle, use an edge $e_i = v_i v_{i+1}$ of the cycle $v_1,\dots,v_{y}$ in $T_{i,j}$ and choose 
any edge $e_*=a_* b_*$ from the cycle to be picked up. By Lemma~\ref{lem:adjtiles}, the four endpoints form a 
clique and therefore, we can go along the edges $v_i a_*$, then the cycle to be picked up (without the edge $e_*$), and 
finally along $b_* v_{i+1}$ to bring us back to the cycle in $T_{i,j}$. To pick up a vertex $a_*$ below, we can just
use the edges $v_i a_*$ and $a_* v_{i+1}$ instead of $v_i v_{i+1}$.

If $y\leq x$, then all edges of the original cycle in $T_{i,j}$ will be used and we end up with $x-y+1 \geq 1$ cycles 
and points below (and including) $T_{i,j}$. If $y>x$, then all cycles and points strictly below $T_{i,j}$ become part 
of the original cycle in $T_{i,j}$ and $y-x>0$ edges of the cycle in $T_{i,j}$ remain unused and part of the final cycle. 
\end{proof}
\begin{figure}
\includegraphics[scale=0.3]{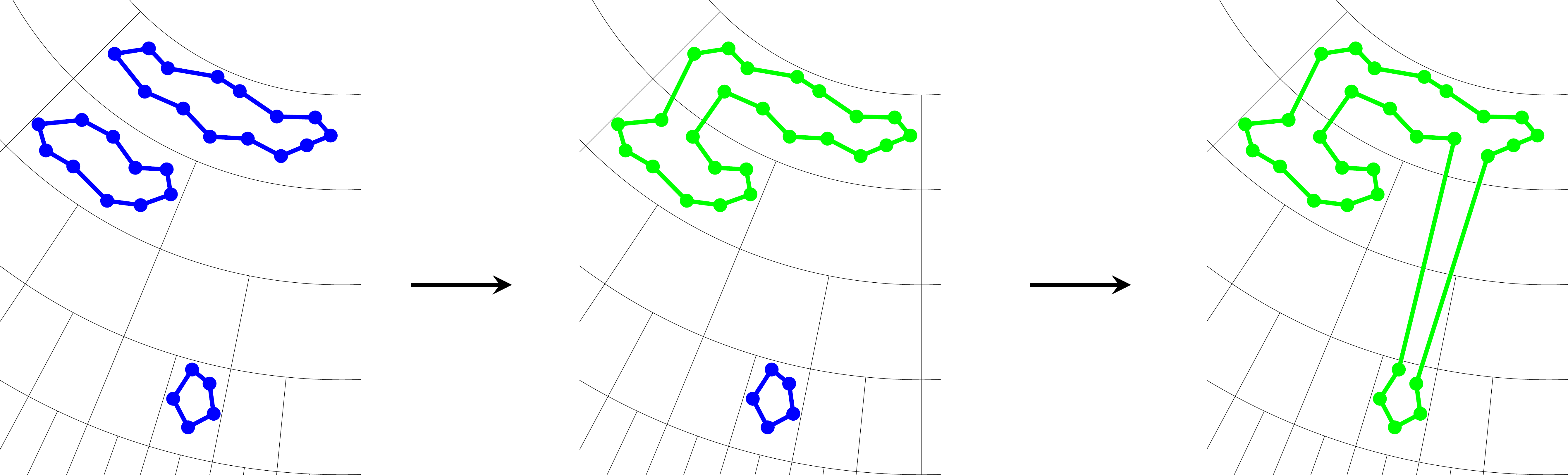}
\caption{In three steps, three cycles (coloured blue) are merged (resulting in the green cycle) by replacing an edge of one cycle 
by a detour around the other cycle. Note that we zoomed into the part of the disk which matters for the cycle merging.}\label{fig:cyclemerge}
\end{figure}

The demand random variables $D_{i,j}$ for admissible $i,j$ are defined in terms of the point counts $N(T_{i,j})$ as follows.
For $i=0$ and $j=0,\dots, n_0-1$ we set:

$$ D_{0,j} = \begin{cases} N(T_{i,j})  & \text{ if } N(T_{i,j}) \in \{1,2\}, \\
              0 & \text{ otherwise, }
             \end{cases}
             $$

\noindent
and, for $0<i\leq i_{\max}$ and $j=0,\dots,n_i-1$ we set:

$$D_{i,j} = \max\{D_{i-1,2j}+D_{i-1,2j+1}+3-N(T_{i,j}),0\}.$$

In particular $D_{i,j}$ and $D_{i,j'}$ are independent for $j\neq j'$, since they depend on disjoint regions. 
Also, the $D_{0,j}$ are i.i.d.~random variables with values in $\{0,1,2\}$ satisfying
$$\pr(D_{0,j}=1)=\mu_0 e^{-\mu_0}, \qquad \pr(D_{0,j}=2) = \frac{\mu_0^2}{2}e^{-\mu_0},$$
where we used the notation $\mu_0:=\im (T_{0,0})\stackrel{Lemma~\ref{lem:exppoints}}{=}\Theta (\nu )$.
%


\begin{lemma}\label{lem:noinneed}
For admissible indices $i,j$, if $D_{i,j} = x$, then the vertices below (and in) $T_{i,j}$ can be covered by at most $x+1$ vertex-disjoint cycles and isolated points (in total).

Moreover, if $i>0$ and $D_{i,j}=0$, then the vertices below (and in) $T_{i,j}$ can be covered by exactly one cycle which has at least one edge which is completely contained in $T_{i,j}$.
\end{lemma}

\begin{proof} The proof is by induction on $i$.
For $i=0$, the claim is clear because then $D_{0,j}=0$ implies that there is either one cycle or no vertex in $T_{i,j}$.

For $i>0$, assuming the claim for $i-1$ we show it for $i$. By the induction hypothesis, the vertices below $T_{i-1,2j}$ ($T_{i-1,2j+1}$, respectively) 
can be covered by $D_{i-1,2j}+1$ ($D_{i-1,2j+1}+1$, respectively) many vertex-disjoint cycles and isolated points. 
Thus, the vertices strictly below $T_{i,j}$ can be covered by $D_{i-1,2j}+D_{i-1,2j+1}+2$ many vertex-disjoint cycles and isolated points in total. 
If $N(T_{i,j}) \geq 3$, then by Lemma~\ref{lem:cyclemerging}, the vertices below $T_{i,j}$ can be covered by

$$ \begin{array}{rcl} 
\max \{1,D_{i-1,2j}+D_{i-1,2j+1}+2-N(T_{i,j})+1\} & \leq & \max\{D_{i-1,2j}+D_{i-1,2j+1}+3-N(T_{i,j}),0\} +1 
\\ & = & D_{i,j}+1.
       \end{array} $$

If $N(T_{i,j}) \leq 2 <3$, then the points in $T_{i,j}$ just remain as single points and the area in and below $T_{i,j}$ can be covered by 
 at most $D_{i-1,2j}+D_{i-1,2j+1}+4 \leq D_{i,j}+1$ vertex-disjoint cycles and isolated points.

In particular, if $D_{i,j}=0$, the points in the area in and below $T_{i,j}$ can be covered by one cycle or point. 
If $i>0$, then the condition $D_{i,j}=0$ and the definition of $D_{i,j}$ imply that there are at least 3 points in 
$T_{i,j}$. Hence, the vertices in and below $T_{i,j}$ can be covered by exactly one cycle, which will have
at least one edge with both endpoints inside $T_{i,j}$ (using that $N(T_{i,j} ) > D_{i-1,2j}+D_{i-1,2j+1}+2$).
\end{proof}

\begin{lemma}\label{lem:detconclu}
If $D_{i,j}=0$ for $i=i_{max}$ and for all $j=0,\dots,n_{i}-1$, then there is a Hamilton cycle.
\end{lemma}

\begin{proof}
Firstly, we observe that for $i=i_{max}$, if $D_{i,j}=0$, then all vertices in and below $T_{i,j}$ can be covered by 
one 
cycle that contains an edge whose endpoints are both in $T_{i,j}$ by Lemma~\ref{lem:noinneed}. Taking such an edge for 
$T_{i,0}$ and $T_{i,1}$, the four endpoints form a clique by the triangle inequality because all radial coordinates are at most 
$\frac{R}{2}$ and hence the cycle of $T_{i,1}$ can be taken as a detour to the cycle of $T_{i,0}$ as in the proof of 
Lemma~\ref{lem:cyclemerging}. 
As a result, we have a cycle covering all vertices below $T_{i,0}$ and $T_{i,1}$ and with an edge inside the $i$-th layer. 
We can repeat this procedure to merge this resulting cycle also with those in $T_{i,2},\dots, T_{i,n_i-1}$. 
We will end up with one cycle covering all vertices below 
all tiles $T_{i,0}, \dots, T_{i,n_i-1}$, and this cycle contains an edge whose endpoints are both in the inner disk with radius 
$\frac{R}{2}$. 
The remaining vertices in the inner disk, that are not in any tile, form a clique and in particular can be covered by a cycle.
We can again merge this cycle with the one we created earlier via the same trick.
\end{proof}

\subsection{Probabilistic lemmas which ensure the a.a.s.~successful termination of the procedure}

In this subsection we show that the algorithm explained previously works successfully for the poissonised KPKVB model 
$\GPo(n;\alpha,\nu)$ 
with $M \stackrel{\Delta}{=} \Po (n)$ many vertices (the standard depoissonisation of Lemma~\ref{lem:Poisson} gives then the result 
in the standard KPKVB model).
Lemma~\ref{lem:exptail} of this section shows the exponential decay of the demand random variables, which we then use in 
Lemma~\ref{lem:zerodemhighup} to show that the demand 
random variables are simultaneously zero in the maximum layer. Appealing to Lemma~\ref{lem:detconclu}, we can then conclude that this 
makes the algorithm work.

\subsubsection{Sub-exponential tail decay of demand}
We first show the following technical lemma:
\begin{lemma}\label{lem:expineq}
For all real $\epsilon \in (0,1)$, there exists $\kappa = \kappa(\epsilon)>0$ such that for all $i \in \N_{>0}$, for all $x \geq \kappa i^2 \ln(1+ i)$ we have: $(x+1)e^{-\frac{x}{i^2}} \leq \epsilon$.
\end{lemma}
\begin{proof}
Pick $\kappa> \max\{\frac{1}{\ln 2},3\}$ such that $(\kappa+1)2^{3-\kappa} \leq \epsilon$; this is possible as 
$\lim_{a \rightarrow \infty} (a+1)2^{3-a} = 0$.
We prove the lemma in the following way: in the first step we verify that for $x = \kappa i^2 \ln(1+i)$ we have $(x+1)e^{-\frac{x}{i^2}} \leq \epsilon$, and then we show that the left-hand side of the inequality is monotone decreasing in $x$ (by showing that its derivative with respect to $x$ is negative). Since the right-hand side is independent of $x$, this clearly implies the lemma.

For the first step, we need to show that $(\kappa i^2 \ln(1+i) +1)e^{-\kappa \ln(1+i)} \leq \epsilon$. Using that $i^2 \leq (1+i)^2$, $\ln(1+i) \leq 1+i$ and $1\leq (1+i)^3$, we note that the left-hand side of the inequality can be bounded from above by 
\begin{align*}
(\kappa i^2 \ln(1+i) +1)e^{-\kappa \ln(1+i)}& \leq (\kappa+1)(1+i)^3 (1+i)^{-\kappa}= (\kappa+1)(1+i)^{3-\kappa}
\end{align*}
Now, if we plug in $i=1$, this upper bound is  at most $\epsilon$ by the choice of $\kappa$. The derivative in 
$i$ of the latter expression is
\[ (\kappa+1) (3-\kappa)(1+i)^{2-\kappa}, \]
which is negative for $\kappa>3$ for all $i\geq 1$, Therefore, the upper bound is monotone decreasing in $i$ and hence, for all $i \ge 1$ and  $x = \kappa i^2 \ln(1+i)$, we have  $(x+1)e^{-\frac{x}{i^2}} \leq \epsilon$, concluding the first step.

For the second step, we need to verify that the derivative of $(x+1)e^{-\frac{x}{i^2}}$ is negative in $x \geq \kappa i^2 \ln(1+i) >0$: using the assumptions of  $\kappa >\frac{1}{\ln 2}$ and $i\geq 1$, we have $$e^{-\frac{x}{i^2}}\left(1+(x+1)\left(-\frac{1}{i^2}\right)\right)\leq 1-\kappa\ln(1+i) - \frac{1}{i^2}\leq 1-\kappa \ln 2 <0.$$
The lemma follows.
\end{proof}


We are now ready to state and prove the main lemma of this section.
 \begin{lemma}\label{lem:exptail}
There is a constant $c>0$ such that for $0<\alpha < \frac{1}{2}$ and 
$\nu$ sufficiently large, for all admissible $i,j$, and all $t\geq 0$: 

$$\Pee(D_{i,j} \geq t) \leq e^{-c t}.$$ 

\noindent
\end{lemma}

\begin{proof}
Set $c = 10$, $c_0 = c + \sum_{i=1}^\infty \frac{1}{i^2}<\infty$ and $c_i = c_{i-1}-\frac{1}{i^2}$, for $i>0$. 
So, in particular, we have $\infty > c_0 > c_1 > \dots > c = 10>0$. 

We prove the lemma by induction on $i$. 
For the base case $i=0$, the claim is clear for $t > 2$ because $D_{0,j} \in \{0,1,2\}$, so $\Pee (D_{0,j} \geq t) = 0 < e^{-c_0 t}$. For $t=1,2$,

\begin{align*}
\Pee (D_{0,j} \geq t) \leq \mu_0 e^{-\mu_0} + \frac{\mu_0^2}{2} e^{-\mu_0} = O(\nu^2) e^{-\Theta(\nu)},
\end{align*}

where the equality follows by Lemma~\ref{lem:exppoints}. In particular, by choosing $\nu$ large enough, it holds that $\Pee (D_{0,j} \geq t) \leq e^{-c_0 t}$ for $t=1,2$.


For the inductive step, assume the statement is true for $i-1$ with $1\leq i \leq i_{max}$.
Note that as $D_{i-1,2j}$ and $D_{i-1,2j+1}$ are independent, we can apply the induction hypothesis to $D_{i-1,2j}$ and $D_{i-1,2j+1}$  to get

\begin{align}\label{claim:dim}
\Pee (D_{i-1,2j}+D_{i-1,2j+1} \geq t) &\leq \sum_{s=0}^t \Pee(D_{i-1,2j} \geq s)\Pee(D_{i-1,2j+1} \geq t-s) \nonumber \\
&\leq \sum_{s=0}^t e^{-c_{i-1}s} e^{-c_{i-1}(t-s)} = (t+1)e^{-c_{i-1}t}.
\end{align}  


Define $$\epsilon = \min\left\{e^{-3 c_0 },\frac{1}{2}(1-e^{-c}) \right\},$$
let $\kappa = \kappa(\epsilon)$ as in Lemma~\ref{lem:expineq} and set

$$ t_i := \kappa i^2 \ln(1+i) +3.$$ 

We make a case distinction in $t$.

\noindent
\underline{Case 1:} $t \geq t_i$.

Using the definition of $D_{i,j}$ and by~\eqref{claim:dim}, we have
\begin{align*}
\Pee(D_{i,j} \geq t) & = \sum_{s=0}^\infty \Pee (N(T_{i,j})=s)\Pee (D_{i-1,2j}+D_{i-1,2j+1} \geq t+s-3) \\
&\stackrel{\eqref{claim:dim}}{\leq} \sum_{s=0}^\infty \Pee (N(T_{i,j}) =s)(t+s-3+1)e^{-c_{i-1}(t+s-3)}. 
\end{align*}
Now, we can apply Lemma~\ref{lem:expineq} to $x=t+s-3 \geq \kappa i^2 \ln(1+i)$ to deduce that
$$(x+1)e^{(c_i-c_{i-1})x} = (x+1)e^{-\frac{x}{i^2}} \leq \epsilon \text{, which implies } (x+1)e^{-c_{i-1}x} \leq \epsilon e^{-c_i x}.$$
We infer that
\begin{align*}
\Pee(D_{i,j} \geq t) &\leq \sum_{s=0}^\infty \Pee(N(T_{i,j})=s) \epsilon e^{-c_i (t+s-3)} \\
&= \epsilon e^{3c_i} e^{-c_i t} \sum_{s=0}^\infty \Pee(N(T_{i,j})=s) e^{-c_i s} \\
& \leq e^{-c_i t} \sum_{s=0}^\infty \Pee(N(T_{i,j})=s) = e^{-c_i t},
\end{align*}
where the third line follows by choice of $\epsilon$ and the definition of the sequence $c_0, c_1, \dots$, 
which implies that $c_i < c_0$.
\medskip

\noindent
\underline{Case 2:} $t<t_i$.

Let $\mu_i: = \im (T_{i,0})$.
We first observe that for all $i \in \N_0$:

\begin{align}\label{eq:muione}
\mu_i \geq   (c_i t_i + \ln 2)\frac{2}{1-\ln 2}  
\end{align}
and
\begin{align}\label{eq:muikappa}
\frac{1}{2}\mu_i \geq t_i+3.
\end{align}

To see that this holds, note that as $\mu_i = \Omega(\nu 2^{i(1-2\alpha)} )$ (see Lemma~\ref{lem:exppoints}), we can take 
any $\nu_*>0$ and then pick $i_0 = i_0(\nu_*) \in \N$ such that for all $\nu \geq \nu_*$ and all $i \geq i_0$, the claims hold 
(as the right-hand side is independent of $\nu$ and grows at most polynomially in $i$ whereas $\mu_i$ grows exponentially in $i$). 
Then, as the right-hand sides of~\eqref{eq:muione} and~\eqref{eq:muikappa} are independent of $\nu$, we can pick $\nu_{**}>\nu_*$ 
large enough such that~\eqref{eq:muione} and~\eqref{eq:muikappa} also hold
for $i=0,\dots,i_0(\nu_*)$. 

Thus,~\eqref{eq:muione} and~\eqref{eq:muikappa} hold for all $\nu > \nu_{**}$ and all $i \in \N_0$.

We have
\begin{align*}
&\Pee(D_{i,j} \geq t)\\
 &= \sum_{j=0}^\infty \Pee(D_{i-1,2j}+D_{i-1,2j+1} = j) \Pee( D_{i-1,2j}+D_{i-1,2j+1}+3-N(T_{i,j}) \geq t | D_{i-1,2j}+D_{i-1,2j+1} = j) \\
&= \sum_{j=0}^\infty \Pee (D_{i-1,2j}+D_{i-1,2j+1} = j) \Pee( N(T_{i,j}) \leq j+3-t) \\
&= \sum_{j=\max\{t-3,0\}}^\infty \Pee (D_{i-1,2j}+D_{i-1,2j+1} = j)\Pee(N(T_{i,j}) \leq  j+3-t) \\ 
&\leq \sum_{j=\max\{t-3,0\}}^\infty \Pee (D_{i-1,2j}+D_{i-1,2j+1}=j)\Pee(N(T_{i,j}) \leq  j+3). 
\end{align*}

We split the sum into two parts: for  $j+3 \leq \frac{1}{2} \mu_i$, we apply Lemma~\ref{lem:Chernoff} with $k=j+3$ and $\mu= \mu_i$ and hence, $\frac{k}{\mu} = \frac{j+3}{\mu_i} \leq \frac{1}{2}$.
Therefore, $H(\frac{k}{\mu}) \geq \frac{1}{2}(1-\ln 2)>0$ and we get 
$$\pr(N(T_{i,j}) \leq j+3) \leq e^{-\mu_i\frac{1}{2}(1-\ln 2)}.$$

By~\autoref{eq:muione}, it follows that
$$e^{-\mu_i \frac{1}{2}(1-\ln 2)} \leq \frac{1}{2}e^{-c_i t_i} \leq  \frac{1}{2}e^{-c_i t}.$$
So, we have for the first part of the sum
\begin{align*}
\sum_{\substack{j=\max(t-3,0)\\j+3 \leq \frac{1}{2} \mu_i}}^\infty \Pee(D_{i-1,2j}+D_{i-1,2j+1} =j)\Pee(N(T_i) \leq j+3) \leq e^{-\mu_i\frac{1}{2}(1-\ln 2)}  \leq \frac{1}{2}e^{-c_i t}.
\end{align*}

For the second part, we have $j +3 > \frac{1}{2}\mu_i$. By~\autoref{eq:muikappa}, $\frac{1}{2}\mu_i \geq t_i+3>t_i$. 
By~\eqref{claim:dim} and Lemma~\ref{lem:expineq} with $x = j \geq \kappa i^2 \ln (1+i)$ it holds that
\begin{align*}
\Pee(D_{i-1,2j}+D_{i-1,2j+1} = j) &\leq \Pee(D_{i-1,2j}+D_{i-1,2j+1} \geq j) \leq (j+1)
e^{-c_{i-1}j}\leq \epsilon e^{-c_i j}.
\end{align*} 
With this, we can also bound from above the second sum:
\begin{align*}
\sum_{\substack{j=\max\{t-3,0\} \\ j+3 > \frac{1}{2}\mu_i}}^\infty \Pee(D_{i-1,2j}+D_{i-1,2j+1} =j)\Pee(N(T_{i,j})\leq j+3) &\leq \sum_{j=t}^\infty \epsilon e^{-c_i j} = e^{-c_i t} \epsilon \sum_{j=0}^\infty (e^{-c_i})^j \\
& = e^{-c_i t} \epsilon \frac{1}{1-e^{-c_i}} \\
& \leq e^{-c_i t} \epsilon \frac{1}{1-e^{-c}}\leq \frac{1}{2}e^{-c_i t}.
\end{align*}
where the last inequality follows from the choice of $\epsilon$ and the fact that $c_i > c$.

By combining both sums, we conclude that also for $t$ as in Case 2, 
$\Pee (D_{i,j} \geq t) \leq \frac{1}{2}e^{-c_i t}+\frac{1}{2}e^{-c_i t} = e^{-c_i t}$, and the lemma follows.
\end{proof}

\subsubsection{Deriving Theorem~\ref{thm:HC}}

Finally, Theorem~\ref{thm:HC} is a result of the following lemma together with Lemma~\ref{lem:detconclu}.
\begin{lemma}\label{lem:zerodemhighup}
Let $0<\alpha < \frac{1}{2}$, $\nu$ sufficiently large. Then 
$$\Pee ( D_{i_{\max},j}=0, \text{ for all }j=0,\dots, n_{i_{\max}}-1) = 1-o\left(n^{-1/2}\right).$$
\end{lemma}

\begin{proof}
First, let us recall that the number of tiles in the $i$th layer is $n_{i} = 2^{4-i+\lfloor \frac{R}{2\ln 2} \rfloor}$. 
For $i = i_{max} = \lceil \frac{0.9 R}{2\ln 2} \rceil$, it follows that $n_{i} = \Theta(2^{\frac{0.1 R}{2\ln 2}})   = \Theta( n^{0.1}) $. 
Furthermore, $\mu_i = \Omega(2^{i_{max} (1-2\alpha)}) = \Omega( n^{0.9 (1-2\alpha)})$. 

We have

\begin{align*}
\Pee (\text{for all }j=0,\dots, n_{i}-1: D_{i,j} = 0) = 1- \Pee(D_{i,j} > 0 \text{ for some }j).
\end{align*}

By the union bound over all tiles in layer $i= i_{\max}$

\begin{align*}
\Pee(D_{i,j}>0 \text{ for some } j) &\leq \sum_{j=0}^{n_{i}-1} \Pee (D_{i,j} >0). 
\end{align*}

Now we observe that if $D_{i-1,2j} \leq \frac{\mu_i}{100}$ and $D_{i-1,2j+1} \leq \frac{\mu_i}{100}$ 
and $N(T_{i,j})\geq \frac{3}{100}\mu_i$ all hold, then $3+D_{i-1,2j}+D_{i-1,2j+1}-N(T_{i,j}) \leq 3 + \frac{2}{100}\mu_i - \frac{3}{100}\mu_i \leq 0$ 
since $\mu_i = \Omega(\nu 2^{i (1-2\alpha)})$. Hence if all three of these conditions hold then $D_{i,j} \leq 0$. 
In other words, if $D_{i,j}>0$ then $D_{i-1,2j} > \frac{\mu_i}{100}$ or $D_{i-1,2j+1}>\frac{\mu_i}{100}$ or $N(T_{i,j})<\frac{3}{100}\mu_i$. 

Therefore,

\begin{align*}
\Pee(D_{i,j}>0) &\leq \Pee\left(D_{i-1,2j} > \frac{\mu_i}{100}\text{ or } D_{i-1,2j+1} > \frac{\mu_i}{100}\text{ or } N(T_{i,j})<\frac{3\mu_i}{100}\right) \\
&\leq \Pee\left(D_{i-1,2j} > \frac{\mu_i}{100}\right)+\Pee\left(D_{i-1,2j+1} > \frac{\mu_i}{100}\right) +\Pee\left(N(T_{i,j})<\frac{3\mu_i}{100}\right).
\end{align*}

For the first two terms, we use Lemma~\ref{lem:exptail}, taking $\nu$ sufficiently large, and for the third term, we apply Lemma~\ref{lem:Chernoff}. 
We get

\begin{align*}
\Pee(D_{i,j} >0) \leq 2 e^{-c\frac{\mu_i}{100}} + e^{-\Omega(\mu_i)} = e^{-\Omega(\mu_i)}.
\end{align*}

Using that $\mu_i = \Omega( n^{0.9(1-2\alpha)})$, it follows that $\Pee (D_{i,j}>0) =  O(e^{-\Omega(n^{0.9(1-2\alpha)})})$. 
Since $n_{i_{max}} = \Theta(n^{0.1})$, we obtain

\begin{align*}
\Pee (D_{i,j}>0 \text{ for some } j) \leq \sum_{j} \Pee(D_{i,j} >0) = O\left(n^{0.1} e^{-\Omega(n^{0.9(1-2\alpha)})}\right) = o\left( n^{-1/2}\right),
\end{align*}

and the lemma follows.
\end{proof}

\bibliographystyle{plain}
\bibliography{Ham}

\def\cprime{$'$}
\begin{thebibliography}{10}

\bibitem{ar:AbdBodeFound}
M.A. Abdullah, M.~Bode, and N.~Fountoulakis.
\newblock Typical distances in a geometric model for complex networks.
\newblock {\em Internet Mathematics}, 1, 2017.

\bibitem{ar:AjKomSz85}
M.~Ajtai, J.~Koml\'os, and E.~Szemer\'edi.
\newblock First occurrence of {H}amilton cycles in random graphs.
\newblock {\em North-Holland Mathematics Studies}, 115:173--175, 1985.

\bibitem{BarAlb}
R.~Albert and A.-L. Barab\'asi.
\newblock Statistical mechanics of complex networks.
\newblock {\em Rev. Mod. Phys.}, 74(1):47--97, 2002.

\bibitem{ar:BBKMW}
J.~Balogh, B.~Bollob\'as, M.~Krivelevich, T.~M\"uller, and M.~Walters.
\newblock Hamilton cycles in random geometric graphs.
\newblock {\em The Annals of Applied Probability}, 21(3):1053--1072, 2011.

\bibitem{BFMgiantEJC}
M.~Bode, N.~Fountoulakis, and T.~M{\"u}ller.
\newblock On the largest component of a hyperbolic model of complex networks.
\newblock {\em Electronic Journal of Combinatorics}, 22(3), 2015.
\newblock Paper P3.24, 43 pages.

\bibitem{BFMconnRSA}
M.~Bode, N.~Fountoulakis, and T.~M{\"u}ller.
\newblock The probability that the hyperbolic random graph is connected.
\newblock {\em Random Structures Algorithms}, 49(1):65--94, 2016.

\bibitem{ar:Bol84}
B.~Bollob\'as.
\newblock The evolution of sparse graphs.
\newblock In {\em Graph Theory and Combinatorics}, pages 35--57, London, 1984.
  Academic Press.

\bibitem{CF16}
E.~Candellero and N.~Fountoulakis.
\newblock Clustering and the hyperbolic geometry of complex networks.
\newblock {\em Internet Mathematics}, 12:2--53, 2016.

\bibitem{DiazMitschePerez}
J.~D{\'{\i}}az, D.~Mitsche, and X.~P{\'e}rez.
\newblock Sharp threshold for {H}amiltonicity of random geometric graphs.
\newblock {\em SIAM J. Discrete Math.}, 21(1):57--65, 2007.

\bibitem{F12}
N.~Fountoulakis.
\newblock On the evolution of random graphs on spaces with negative curvature.
\newblock {\em ArXiv e-prints}, May 2012.

\bibitem{ar:FounMull2017}
N.~Fountoulakis and T.~M\"uller.
\newblock Law of large numbers in a hyperbolic model of complex networks.
\newblock {\em Annals of Applied Probability}, 28:607--650, 2018.

\bibitem{ar:FriedKrohmerDiam}
T.~Friedrich and A.~Krohmer.
\newblock On the diameter of hyperbolic random graphs.
\newblock {\em SIAM J. Disc. Math.}, 32:1314--1334, 2018.

\bibitem{ar:AlanXavier}
A.~Frieze, X.~P\'{e}rez-Gim\'{e}nez, P.~Pra\l{}at, and B.~Reiniger.
\newblock Perfect matchings and hamiltonian cycles in the preferential
  attachment model.
\newblock {\em Random Structures and Algorithms, to appear}.

\bibitem{Gugel}
L.~Gugelmann, K.~Panagiotou, and U.~Peter.
\newblock Random hyperbolic graphs: degree sequence and clustering.
\newblock Preprint. Available from \url{http://arxiv.org/abs/1205.1470}.
  Conference version in ICALP 2012.

\bibitem{ar:Gugel}
L.~Gugelmann, K.~Panagiotou, and U.~Peter.
\newblock Random hyperbolic graphs: Degree sequence and clustering.
\newblock In {\em Proceedings of the 39th International Colloquium Conference
  on Automata, Languages, and Programming - Volume Part II}, ICALP'12, pages
  573--585, Berlin, Heidelberg, 2012. Springer-Verlag.

\bibitem{KM18}
M.~Kiwi and D.~Mitsche.
\newblock Spectral gap of random hyperbolic graphs and related parameters.
\newblock {\em Annals of Applied Probability}, 28:941--989, 2018.

\bibitem{KiwiMit}
M.~A. Kiwi and D.~Mitsche.
\newblock A bound for the diameter of random hyperbolic graphs.
\newblock In Robert Sedgewick and Mark~Daniel Ward, editors, {\em Proceedings
  of the Twelfth Workshop on Analytic Algorithmics and Combinatorics, {ANALCO}
  2015, San Diego, CA, USA, January 4, 2015}, pages 26--39. {SIAM}, 2015.

\bibitem{KiwiMit2017+}
M.~A. Kiwi and D.~Mitsche.
\newblock On the second largest component of random hyperbolic graphs.
\newblock Available at \texttt{arXiv:1712.02828v1}, 2017.

\bibitem{ar:KomSzem83}
J.~Koml\'os and E.~Szemer\'edi.
\newblock Limit distribution for the existence of {H}amiltonian cycles in a
  random graph.
\newblock {\em Discrete Math.}, 43:55--63, 1983.

\bibitem{ar:Korshunov77}
A.D. Korshunov.
\newblock A solution of a problem of {P}. {E}rd{\H{o}}s and {A}. {R}{\'e}nyi
  about {H}amilton cycles in non-oriented graphs.
\newblock {\em Metody Diskr. Anal. Teoriy Upr. Syst., Sb. Trudov Novosibirsk},
  31:17--56 (in Russian), 1977.

\bibitem{ar:Krioukov}
D.~Krioukov, F.~Papadopoulos, M.~Kitsak, A.~Vahdat, and M.~Bogu{\~n}{\'a}.
\newblock Hyperbolic geometry of complex networks.
\newblock {\em Phys. Rev. E (3)}, 82(3):036106, 18, 2010.

\bibitem{bk:LastPenrose}
G.~Last and M.~Penrose.
\newblock {\em Lectures on the Poisson Process}.
\newblock IMS Textbooks. Cambridge University Press, 2018.

\bibitem{MPW11}
T.~M{\"{u}}ller, X.~P{\'{e}}rez{-}Gim{\'{e}}nez, and N.C. Wormald.
\newblock Disjoint hamilton cycles in the random geometric graph.
\newblock {\em Journal of Graph Theory}, 68(4):299--322, 2011.

\bibitem{ar:MullerDiam}
T.~M\"uller and M.~Staps.
\newblock The diameter of {KPKVB} random graphs.
\newblock Available at \texttt{arXiv 1707.09555}, 2017.

\bibitem{Penroseboek}
M.~D. Penrose.
\newblock {\em Random Geometric Graphs}.
\newblock Oxford University Press, Oxford, 2003.

\bibitem{ar:Posa76}
L.~P\'osa.
\newblock Hamiltonian circuits in random graphs.
\newblock {\em Discrete Math.}, 14:359--364, 1976.

\bibitem{ar:Nick}
R.W. Robinson and N.C. Wormald.
\newblock Almost all regular graphs are hamiltonian.
\newblock {\em Random Structures and Algorithms}, 5:363--374, 1994.

\bibitem{stillwellboek}
J.~Stillwell.
\newblock {\em Geometry of surfaces}.
\newblock Universitext. Springer-Verlag, New York, Berlin, Heidelberg, 1992.

\end{thebibliography}

\end{document}